\documentclass[a4paper,11pt,twoside,leqno]{amsart}
\usepackage{latexsym}
\usepackage[all]{xy}

\usepackage{amssymb} 
\usepackage{amsmath} 
\usepackage{amscd}
\usepackage{amsthm}
\usepackage[usenames,dvipsnames]{xcolor}
\usepackage{pgf,tikz}
\usepackage{tkz-tab}
\usetikzlibrary{arrows}
\usepackage{enumerate}
\usepackage{hyperref}
\usepackage{caption}
\usepackage{subcaption}

\def\P{{\mathbb{P}}}
\def\PD{\check{\mathbb P}}
\def\Z{{\mathbb{Z}}}
\def\R{{\mathbb{R}}}
\def\A{{\mathcal{A}}}
\def\B{{\mathcal{B}}}

\newtheorem{theorem}{Theorem}
\newtheorem{prop}{Proposition}
\newtheorem{cor}{Corollary}
\newtheorem{lemma}{Lemma}

\theoremstyle{definition}
\newtheorem{conj}{Conjecture}
\newtheorem{define}{Definition}
\newtheorem{rem}{Remark}

\title{Logarithmic bundles of deformed Weyl arrangements of type $A_2$}
\author{Takuro Abe}
\email{abe.takuro.4c@kyoto-u.ac.jp}
\urladdr{\url{https://www.math.kyoto-u.ac.jp/~abetaku/}}
\address{Department of Mechanical Engineering and Science,
Kyoto University,
Sakyo-ku, Kyoto, 606-8501,
Japan}

\author{Daniele Faenzi}
\email{daniele.faenzi@univ-pau.fr}
\urladdr{{\url{http://univ-pau.fr/~faenzi/}}}

\author{Jean Vall\`es}
\email{jean.valles@univ-pau.fr}
\address{Universit\'e de Pau et des Pays de l'Adour,
  Avenue de l'Universit\'e - BP 576 - 64012 PAU Cedex - France}
\urladdr{\url{http://web.univ-pau.fr/~jvalles/jean.html}}


\thanks{T. A. is supported by JSPS Grants-in-Aid for Young Scientists (B) No. 24740012. D. F. and J. V. partially supported by GEOLMI
  ANR-11-BS03-0011. All authors supported by Sakura Campus France project {\it G\'eom\'etrie, combinatoire et topologie des
arrangements d'hyperplans}}

\begin{document}

\maketitle

\begin{abstract}
We consider  deformations of the Weyl arrangement of 
type $A_2$, which include 
the extended Shi and Catalan arrangements. These last ones are
well-known to be free. We study their sheaves of logarithmic vector
fields in all other cases, and show that they are Steiner bundles. Also, we determine 
explicitly their unstable lines. As a corollary, some counter-examples
to the shift isomorphism problem are given.
\end{abstract}

\section*{Introduction}
Let $\Phi$ be an irreducible crystallographic root system in Euclidean
space $V \simeq \R^m$, let $\Phi^+ \subset \Phi$ be the positive roots, and let $\eta$ be the Coxeter number
of $\Phi$.
Let $x_1,\ldots,x_m$ be coordinates of $V$, set
$S=\R[x_0,\ldots,x_m]$, and denote by $\mathrm{Der}(S)$ the free $S$-module of
derivations of $S$,  generated by the partial derivatives
$\frac{\partial}{\partial x_0}, \ldots,\frac{\partial}{\partial x_m}$.
For $s\in \Z$ and $\alpha \in \Phi^+$, define the hyperplanes:
$$H_{\alpha,s}=\{x \in \P^{m} \mid \alpha(x_1,\ldots,x_m)=s x_0\}
\subset \P^m.$$
Fix integers 
$k,j \ge 0$, and define the (cone over the) \textit{deformation of the Weyl arrangement of type 
$\Phi$}:
$$
\A_\Phi^{[-j,k+j]}= \{x_0=0\} \cup \{H_{\alpha,s} \mid \alpha \in \Phi^+,
-j \le s \le k+j\}.
$$

The combinatorics, topology and algebra of $\A = \A_{\Phi}^{[-j,k+j]}$
have been studied by several authors, for
instance by Postnikov and Stanley in \cite{PS},  
by Athanasiadis in \cite{Ath0}, by Edelman and Reiner in \cite{ER}, and 
by Yoshinaga in \cite{Y0}, especially when $k \in \{0,1\}$. In particular, 
the freeness of  $\A$ when $k=0, 1$ was conjectured by Edelman and Reiner and proved in \cite{Y0} by Yoshinaga. 
By {\it freeness} here we mean freeness of logarithmic derivation module
of $\A$:
$$
D_0(\A):=\{\theta \in \mathrm{Der}(S) \mid
\theta(f_{j,k})=0\},
$$
where $f_{j,k}$ is the form of degree $n = |\A|$ given as product of linear forms defining the hyperplanes of
$\A$.
Equivalently, freeness means splitting of the
the sheafification $T_{\A}$ of $D_0(\A)$. This is a
reflexive sheaf of rank $m$ called  \textit{logarithmic sheaf}. 
It can also be defined as the kernel of the
 Jacobian map:
$$\begin{CD}\mathcal{O}_{\P^m}^{m+1} @>\nabla(f_{j,k})>> \mathcal{O}_{\P^m}(n-1).\end{CD}$$

In spite of the good knowledge of $T_\A$ for $k\in \{0,1\}$, almost
nothing is known about $T_\A$ for $k \ge 2$, not even for $A_2$.
For example, setting $\B = \A_\Phi^{[-j-1,k+j+1]}$, the \textit{shift
  isomorphism problem}, cf. \cite[Remark 3.7]{Y} asks
whether there is an isomorphism:
\begin{eqnarray}
T_{\A} &\simeq& T_{\B}(\eta). \label{eq2}
\end{eqnarray}
Another question is the \textit{shifted dual isomorphism problem}, to
the effect that:
\begin{eqnarray} \label{dualiso}
T_{\A} &\simeq& T_{\A}^\vee(-\eta(k+2j+1)). \label{eq3}
\end{eqnarray}

These  isomorphisms hold when $k=0,1$ by \cite{Y0}. However, even equality of 
characteristic polynomials (i.e., of Chern
classes) of these sheaves is unknown in general: this is  the so-called ``functional
equation'' conjecture of \cite{PS}, cf. also \cite[Conjecture 3.4 and 3.5]{Y}.
However the roots of the characteristic polynomial should have real
part $\eta(k+2j+1)/2$ by the ``Riemann hypothesis'' of \cite{PS},
verified for $\Phi$ of type $A,B,C,D$ in \cite{Ath1}.

\medskip

In this paper, we are most interested in
 the case $\Phi=A_2$.  We switch to
the notation $(z,x,y)$ rather than $(x_0,x_1,x_2)$, and we fix
  $\A=\A_{A_2}^{[-j,k+j]}$. We have $\eta=3$.
In this case $T_\A$ is locally free  (a vector bundle) of rank $2$, and
the lines of $\A$ are defined by vanishing of the form:
$$f_{j,k}=z\prod_{-j \le s \le k+j}(x-s z)(y-s z)(y+x-s z).$$
Concerning resolutions, our
main theorem is the following.
\begin{theorem} \label{main1}
For any $k \ge 2$ and $j \ge 0$, there is a resolution:
$$
0 \to \mathcal{O}_{\P^2}(-1)^{k-1 }
\to \mathcal{O}_{\P^2}^{k+1}
\to T_{\A}(2k+1+3j) \to 0.
$$
In particular, $T_{\A}(2k+1+3j)$ is a Steiner bundle.
\label{main2}
\end{theorem}
By \textit{Steiner bundle} here we mean a vector bundle whose
resolution is given by a matrix of linear forms.
This agrees and gives a new interpretation of the following formulas,
easily obtained for instance counting multiple points and using
\cite[Remark 2.2]{FV}:
$$
\mbox{$
c_1(T_{\A}(2k+3j+1))=k-1, \qquad c_2(T_{\A}(2k+3j+1))=\frac{k(k-1)}{2}$}.
$$ 
Since $T_\A(2k+1+3j)$ is a Steiner bundle, for any
line $L \subset \P^2$, by restriction onto $L$ we get a surjective map:
$$\begin{CD}
     \mathcal{O}_L^{k+1} @>>> T_{\A}(2k+1+3j)|_L \end{CD}.$$ 
This implies that $T_{\A}(2k+1+3j)|_L = \mathcal{O}_L(a)\oplus  \mathcal{O}_L(k-1-a)$ with $0\le a\le k-1$.
When $a=0$ or $a=k-1$ the number $|k-1-2a|$ is as large as possible.
This justifies the next definition, cf. \cite[Page 508]{V} or \cite[Definition 2.1]{FMV}.
\begin{define}
 Let $k\ge 2$  and $E$ be a Steiner bundle defined by:
$$\begin{CD}
  0@>>>  \mathcal{O}_{\P^2}(-1)^{k-1} @>>> \mathcal{O}_{\P^2}^{k+1}@>>> E @>>>0.
  \end{CD}
$$
A line $L$ such that $E|_L=\mathcal{O}_L\oplus \mathcal{O}_L(k-1)$ or
equivalently $H^0(\P^2,E^{\vee}|_L)\neq 0$ or equivalently
$H^1(L,E|_L(-2))\neq 0$ is called \textit{unstable}. The set of such
lines is denoted by $W(E)$, it is naturally a subscheme of $\PD^2$.
 These unstable lines were first called \textit{superjumping lines} in \cite{DK}.
\end{define}

Our next result, tightly related with Theorem \ref{main2}, deals with the set of  unstable lines of 
$T_{\A}$, which we can  determine explicitly. The figure shows
them in case $j=0$ and $k=3$ or $k=4$, the thick orange lines being
unstable (the solid ones lie in the arrangement, the dashed ones don't).

\newcommand{\boundellipse}[3]
{(#1) ellipse (#2 and #3)
}

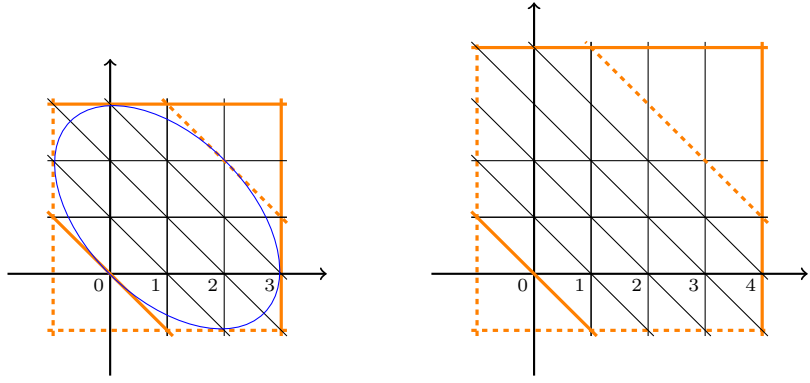
\begin{figure}[h!]
  \begin{subfigure}[b]{0.5\textwidth}
    \centering
  \begin{tikzpicture}[scale=0.75]
    \draw[color = orange,very thick,dash pattern=on 2pt off 2pt] (-1,-1.1) -- (-1,3.1);
    \draw[thick,->] (0,-1.8) -- (0,3.8);
    \draw (-0.2,-0.2) node {\tiny{$0$}};
    \draw (0.8,-0.2) node {\tiny{$1$}};
    \draw (1.8,-0.2) node {\tiny{$2$}};
    \draw (2.8,-0.2) node {\tiny{$3$}};
    \draw (1,-1.1) -- (1,3.1);
    \draw (1,-1.1) -- (1,3.1);
    \draw (2,-1.1) -- (2,3.1);
    \draw[color = orange,very thick] (3,-1.1) -- (3,3.1);
    \draw[color = orange,very thick,dash pattern=on 2pt off 2pt] (-1.1,-1) -- (3.1,-1);
    \draw (-1.1,1) -- (3.1,1);
    \draw[thick,->] (-1.8,0) -- (3.8,0);
    \draw (-1.1,1) -- (3.1,1);
    \draw (-1.1,2) -- (3.1,2);
    \draw[color = orange,very thick] (-1.1,3) -- (3.1,3);
    \draw[color = orange,very thick] (1.1,-1.1) -- (-1.1,1.1);
    \draw (2.1,-1.1) -- (-1.1,2.1);
    \draw (3.1,-1.1) -- (-1.1,3.1);
    \draw (3.1,-0.1) -- (-0.1,3.1);
    \draw[color = orange,very thick,dash pattern=on 2pt off 2pt] (3.1,0 .9) -- (0 .9,3.1);
\draw[rotate around = {45:(1,1)},color=blue] \boundellipse{1,1}{1.4241}{2.4};
  \end{tikzpicture}
  \subcaption*{Conic of unstable lines for $j=0$, $k=3$.}   
  \end{subfigure}
\hspace{0.0cm}
  \begin{subfigure}[b]{0.4\textwidth}
    \centering
  \begin{tikzpicture}[scale=0.75]
    \draw[color = orange,very thick,dash pattern=on 2pt off 2pt] (-1,-1.1) -- (-1,4.1);
    \draw[thick,->] (0,-1.8) -- (0,4.8);
    \draw (-0.2,-0.2) node {\tiny{$0$}};
    \draw (0.8,-0.2) node {\tiny{$1$}};
    \draw (1.8,-0.2) node {\tiny{$2$}};
    \draw (2.8,-0.2) node {\tiny{$3$}};
    \draw (3.8,-0.2) node {\tiny{$4$}};
    \draw (1,-1.1) -- (1,4.1);
    \draw (1,-1.1) -- (1,4.1);
    \draw (2,-1.1) -- (2,4.1);
    \draw (3,-1.1) -- (3,4.1);
    \draw[color = orange,very thick] (4,-1.1) -- (4,4.1);
    \draw[color = orange,very thick,dash pattern=on 2pt off 2pt] (-1.1,-1) -- (4.1,-1);
    \draw (-1.1,1) -- (4.1,1);
    \draw[thick,->] (-1.8,0) -- (4.8,0);
    \draw (-1.1,1) -- (4.1,1);
    \draw (-1.1,2) -- (4.1,2);
    \draw[color = orange,very thick] (-1.1,4) -- (4.1,4);
    \draw[color = orange,very thick] (1.1,-1.1) -- (-1.1,1.1);
    \draw (2.1,-1.1) -- (-1.1,2.1);
    \draw (4.1,-1.1) -- (-1.1,4.1);
    \draw (3.1,-1.1) -- (-1.1,3.1);
    \draw (4.1,-0.1) -- (-0.1,4.1);
    \draw[color = orange,very thick,dash pattern=on 2pt off 2pt] (4.1,0 .9) -- (0 .9,4.1);
  \end{tikzpicture}
  \subcaption*{Unstable lines for $j=0$, $k=4$.}   
  \end{subfigure}
\end{figure}

\begin{theorem}
Assume that $k \ge 3$. Then the following six lines
are unstable for the vector bundle $T_{\A}$:
\[
\begin{array}{lllll}
 \mbox{lines of $\A$:} && x=(k+j)z, & y=(k+j)z, & y+x=-j z,\\
\mbox{lines not of $\A$:} && x=-(j+1)z, & y=-(j+1)z, & y+x=(k+j+1)z.
\end{array}
\]
If $k\ge 4$, then 
there are no other 
unstable lines for $T_{\A}$. If $k=3$ these lines are tangent to the smooth conic
\[
C_j := \{3 j^{2} z^{2}+12 j z^{2}-4 x^{2}-4 x y-4 y^{2}+12
      x z+12 y z=0\},
\]
 and a line is unstable for $T_{\A}$ if and only it is tangent to $C_j$.
\label{main3}
\end{theorem}

\begin{figure}[h!]
  \begin{subfigure}[b]{0.7\textwidth}
    \centering
  \begin{tikzpicture}[scale=0.6]
    \draw (-1,-3.1) -- (-1,5.1);
    \draw[thick,->] (0,-3.8) -- (0,5.8);
    \draw (-0.2,-0.2) node {\tiny{$0$}};
    \draw (0.8,-0.2) node {\tiny{$1$}};
    \draw (1.8,-0.2) node {\tiny{$2$}};
    \draw (2.8,-0.2) node {\tiny{$3$}};
    \draw (3.8,-0.2) node {\tiny{$4$}};
    \draw (4.8,-0.2) node {\tiny{$5$}};
    \draw (1,-3.1) -- (1,5.1);
    \draw (4,-3.1) -- (4,5.1);
    \draw (5,-3.1) -- (5,5.1);
    \draw (-2,-3.1) -- (-2,5.1);
    \draw (-3,-3.1) -- (-3,5.1);
    \draw (1,-3.1) -- (1,5.1);
    \draw (2,-3.1) -- (2,5.1);
    \draw (3,-3.1) -- (3,5.1);
    \draw (-3.1,-1) -- (5.1,-1);
    \draw (-3.1,1) -- (5.1,1);
    \draw[thick,->] (-3.8,0) -- (5.8,0);
    \draw (-3.1,1) -- (5.1,1);
    \draw (-3.1,2) -- (5.1,2);
    \draw (-3.1,3) -- (5.1,3);
    \draw (-3.1,4) -- (5.1,4);
    \draw (-3.1,5) -- (5.1,5);
    \draw (-3.1,-2) -- (5.1,-2);
    \draw (-3.1,-3) -- (5.1,-3);
    \draw (3.1-2,-3.1) -- (-3.1,3.1-2);
    \draw (3.1-1,-3.1) -- (-3.1,3.1-1);
    \draw (3.1,-3.1) -- (-3.1,3.1);
    \draw (2.1+2,-1.1-2) -- (-1.1-2,2.1+2);
    \draw (3.1+2,-1.1-2) -- (-1.1-2,3.1+2);
    \draw (3.1+2,-0.1-2) -- (-0.1-2,3.1+2);
    \draw (3.1+2,0 .9-2) -- (0 .9-2,3.1+2);
    \draw (3.1+2,0.9-1) -- (0.9-1,3.1+2);
    \draw (3.1+2,0.9) -- (0.9,3.1+2);
\draw[rotate around = {45:(1,1)},color=blue]\boundellipse{1,1}{1.424}{2.4};
\draw[rotate around = {45:(1,1)},color=blue] \boundellipse{1,1}{1.5*1.4241}{3.6};
\draw[rotate around = {45:(1,1)},color=blue] \boundellipse{1,1}{1.98*1.4241}{4.85};
  \end{tikzpicture}
  \subcaption*{Conics $C_j$ of unstable lines for $k=3$, $j=0,1,2$.}   
  \end{subfigure}
\end{figure}
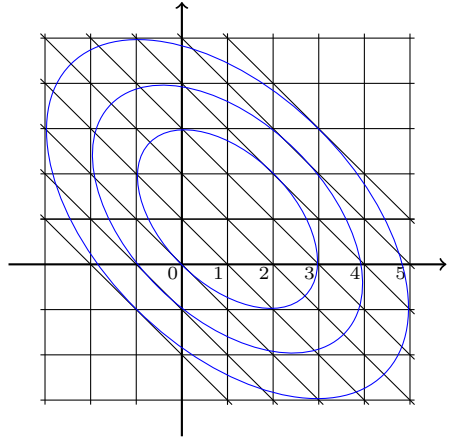

\begin{rem}
 When $k=2$, for any $j$
we will see that  $T_{\A}\simeq
T_{\P^2}(-6-3j)$. Hence all lines are unstable in this case, cf. \S \ref{behave}.
\end{rem}

As a corollary of Theorem \ref{main2} and \ref{main3}, we prove
that the shift isomorphism does not hold even for $A_2$ when $k\ge 3$. 

\begin{cor} \label{non}
When the root system $\Phi$ is of the type $A_2$ the shift isomorphism (\ref{eq2}) holds if and only if 
$k=0,1$ or $2$ i.e. it does not hold for $k \ge 3$.
\end{cor}

Let us go back to more general irreducible crystallographic root systems.
Let $\Phi$ be such a system, set $\A=\A_{\Phi}^{[-j,k+j]}$ and call {\it graded Betti
  numbers} of $T_\A$ the integers $\beta_{i,j}$ appearing in a minimal graded free resolution:
\[
0 \to \oplus_i S(-i-d)^{\beta_{d,i}} \to \cdots \to \oplus_i S(-i)^{\beta_{0,i}} \to D_0(\A) \to 0.
\]
Our results
together with experimental computation with the computer algebra
package Macaulay2 (\cite{M2}) suggest the following.

\begin{conj} \label{all} Let $k \ge 1$, $j \ge 0$,  $\A =
  \A_\Phi^{[-j,k+j]}$ and $\B = \A_\Phi^{[-j-1,k+j+1]}$.
\begin{enumerate}[i)]
  \item \label{1}  The graded Betti numbers of  $T_{\A}$ and
    $T_\B(\eta)$ are the same.
  \item \label{2}   The graded Betti numbers of  $T_{\A}^\vee(-\eta(k
    + 2j +1))$ and $T_{\A}$ are the same.
  \item \label{3}  The projective dimension of $T_{\A}$ is $\min(m-1,k-1)$.
   \item \label{4} The sheaf $T_{\A}$ has a linear resolution if 
     $\Phi=A_m$.
  \end{enumerate}
\end{conj}

Linear resolution here means that $\beta_{j,i}=0$ for all $j$ except for a single
$i=i_0$. The projective dimension is the greatest $j$ with
$\beta_{j,i} \ne 0$ for some $i$.

Theorem \ref{main1} proves parts \eqref{1} and \eqref{4} for
$\Phi=A_2$, while \eqref{2} and \eqref{3} are true for any root system
of rank $2$, so the above conjecture holds for $\Phi = A_2$.

\begin{rem}
  Part \eqref{2} of the above conjecture refines the shifted dual isomorphism
  problem. Indeed, a computer-aided argument shows that the
  isomorphism \eqref{eq3} does not exist in general, even for $\Phi = A_3$, cf. \S\ref{experiments}.
\end{rem}


The paper is organized  as follows. 
In \S \ref{preliminaries} we introduce several results and definitions which will be 
used to prove Theorems \ref{main2} and \ref{main3}. In  Theorem
\ref{main0} we propose a new criterion to determine  the splitting
type  
of $\A^{[-j,k+j]}$ on a line. In \S \ref{resolution} we prove Theorem \ref{main2}. In
\S \ref{unstable} we show Theorem \ref{main3}.
In \S\ref{experiments} we give some
 computer-aided determination of jumping lines and unstable hyperplanes.

\section{Preliminaries}
\label{preliminaries}

In this section let us recall several results on 
arrangements of hyperplanes and the related algebraic geometry.


\subsection{Vector bundles on the projective plane}

In this subsection we review some results and definitions on rank two vector bundles over 
the projective plane. The main reference here is \cite{OSS}. 
We write the Chern classes of a sheaf $E$ as integers $(c_1,c_2)$, meaning $c_1(E)=c_1
H$ and $c_2(E) = c_2 H^2$, $H$ being the class of a line and $H^2$ the class
of a point in $\P^2$.

\subsubsection{Splitting type of bundles}

Let $E$ be a rank two vector bundle on $\P^2$. 
First of all, recall that the 
restriction of any vector bundle $E$ to any projective line $L \subset \P^2$ splits as a direct 
sum of line bundles.

If $E$ has rank $2$, there are two 
integers $a_{1,L}\le a_{2,L}$  such that:
$$E|_L \simeq \mathcal{O}_{L}(a_{1,L})\oplus \mathcal{O}_{L}(a_{2,L}),
\qquad \mbox{with $a_{1,L}+a_{2,L}=c_1(E)$}. $$  
We call 
$(a_{1,L},a_{2,L})$ the \textit{splitting type} of $E$ with respect to $L$. 

By the semi-continuity theorem, these integers
do not change when the line $L$ is chosen a non-empty Zariski open set $U \subset
\PD^2$.
This means that 
there exist two integers $a_1$ and $a_2$ 
such that on any line $L \in U$, the restricted bundle $E|_L=\mathcal{O}_L(a_1)\oplus \mathcal{O}_L(a_2)$. 
Moreover, $|a_1-a_2|=\mbox{min}_{L \in  \PD^2} 
\{|a_{1,L}-a_{2,L}|\}.$ The couple $(a_1,a_2)$ is called the \textit{generic splitting type} of $E$.

\subsubsection{Stability of bundles and jumping lines}

Let $E$ be a vector bundle of rank $2$ on $\P^2$.
We can tensor $E$ by 
$\mathcal{O}_{\P^2}(t)$ in such a way that 
$c_1(E(t)) \in \{-1,0\}$. We call $E(t)$ the normalized twist of $E$. Then $E$ is \textit{stable} if
and only if $H^0(\P^2,E(t))=0$.
If  $c_1(E(t))=0$, 
and $H^0(\P^2,E(t-1))=0$ then $E$ is \textit{semistable} and $E$ is
\textit{strictly semistable} if it is semistable, but not stable.

By the Grauert-M\"{u}lich theorem, 
if $E$ is stable (or semi-stable), then 
the generic splitting type of $E$ verifies $|a_1-a_2|\le 1$. This leads to the notion of jumping lines.

\begin{define}
 A line $L$ such that 
$E|_L=\mathcal{O}_L(a_1-k)\oplus \mathcal{O}_L(a_2+k)$, where 
$k>0$ and $(a_1,a_2)$ with $a_1\le a_2$ is the  generic splitting type
of a 
rank-two semi-stable vector bundle $E$,
 is a \textit{jumping line of order $k$} of $E$. The scheme of jumping lines lives in $\PD^2$ and it is denoted by $S(E)$.
A jumping line of order $k>1$ is a singular point in $S(E)$. 
\end{define}

\subsection{Steiner bundles and unstable lines}
\label{steiner-reduction-extension}
Let $i \ge 2$ and let us consider a Steiner bundle $F_i$:
$$\begin{CD}
  0@>>>  \mathcal{O}_{\P^2}(-1)^{i-2} @>>> \mathcal{O}_{\P^2}^{i}@>>> F_i @>>>0.
  \end{CD}
$$
Then one can construct a new Steiner bundle from $F_i$ in two different ways.
\subsubsection{By reduction} \label{reduction}
 Assume that $i\ge 4$ and that the line $L$ is unstable for $F_i$. Then the kernel $F_{i-1}$ of the map $F_i\twoheadrightarrow \mathcal{O}_L$  is again a Steiner bundle
with resolution:  $$\begin{CD}
  0@>>>  \mathcal{O}_{\P^2}(-1)^{i-3} @>>> \mathcal{O}_{\P^2}^{i-1}@>>> F_{i-1} @>>>0. 
  \end{CD}
$$ The proof is done in \cite[Proposition 2.1]{V}. Moreover in the same
proposition it is proved that $W(F_{i})\setminus \{L\}
\subset W(F_{i-1})$.
\subsubsection{By extension} \label{extension}
Let $H \subset \P^2$ be any line. A non-trivial extension
$$\begin{CD}
  0@>>>   F_{i} @>>> F @>>> \mathcal{O}_{H} @>>>0 
  \end{CD}
$$
is a also a Steiner bundle with resolution:
 $$\begin{CD}
  0@>>>  \mathcal{O}_{\P^2}(-1)^{i-1} @>>> \mathcal{O}_{\P^2}^{i+1}@>>> F @>>>0. 
  \end{CD}
$$
Indeed, the surjection $\mathcal{O}_{\P^2} \twoheadrightarrow
\mathcal{O}_H$ lifts to $\mathcal{O}_{\P^2} \to F$ because
$H^1(\P^2,F_i)=0$.
Combining this with the surjection $\mathcal{O}^i_{\P^2}
\twoheadrightarrow F_i$ we obtain the required epimorphism 
$\mathcal{O}^{i+1}_{\P^2} \twoheadrightarrow F$, whose kernel is an
extension of $\mathcal{O}_{\P^2}^{i-2}(-1)$ by
$\mathcal{O}_{\P^2}(-1)$, and therefore precisely $\mathcal{O}_{\P^2}^{i-1}(-1)$.

Moreover, when $i\ge 3$ the line $H$ is unstable for $F$ and again
according to \cite[Proposition 2.1]{V} we know that
$W(F)\setminus \{H\} \subset W(F_i)$. We
will use these two constructions later on.

\subsection{Behaviour of first Steiner bundles}
\label{behave}

We give here a quick overview of jumping and unstable lines for Steiner for
low $i$.

\begin{itemize}
\item[$i=2$.] In this case we have $F_{2} \simeq \mathcal{O}_{\P^2}^2$.
\item[$i=3$.] It is well-known that $F_{3} \simeq T_{\P^2}(-1)$.
All lines are unstable and none of them is a jumping line.
\item[$i=4$.] Unstable and jumping lines coincide this time,
i.e. $S(F_4)=W(F_4)$.
By \cite{DK}, $W(F_4)$ is a smooth conic in $\PD^2$, and the unstable
lines of $F_4$ are the tangent lines to the dual conic.
\item[$i=5$.] Also this time we have $S(F_5)=W(F_5)$. The
scheme $W(F_5)$ is either finite of length $6$ or consist of a smooth
conic in $\PD^2$, see \cite{DK}.
\item[$i\ge 6$.] 
Unstable lines do not always exist in this range. When they do, they are jumping lines of maximal 
order, namely the splitting type on an unstable line $H$ is $(0,i-2)$. 
The scheme $W(F_5)$ is either finite of length $\le i+1$ or consist of a smooth
conic in $\PD^2$, see \cite{DK,V}.
\end{itemize}

\subsection{Line arrangements and vector bundles}

Let $\A$ be a line arrangement in $\P^2$ and let be $H$ a line of $\A$. 
Define $n=|\A|$ and the restricted arrangement of points $\A^H:=\{K
\cap H \mid K \in \A,\ K \neq H\}$.
Set $h:=|\A^H|$.

Let $t_{\A,H,i}$ be the number of points with multiplicity $i$ on $H$. 
The ``number of triple points'' on $H$ is:
\[
\mbox{$t_{\A,H}=\sum_{i\ge 3}(i-2)t_{\A,H,i}.$}
\]
We recall the following result, 
which is often used in the rest of this article.
\begin{prop}[\cite{FV}, Proposiont 5.1] \label{FV}
There is an 
exact sequence
$$ 0 \to T_{\A} \to T_{\A \setminus \{H\}} \to
\mathcal{O}_H(-t_{\A,H}) \to 0.
$$
\end{prop}

\begin{lemma} \label{lemma}
  We have $t_{\A,H}=n-1-h$.
\end{lemma}

\begin{proof}
It is clear that $h=n-1$ when there are only double points on $H$. 
When there is a point of multiplicity $i\ge 3$ on $H$, 
it is necessary to remove 
$i-2$ to $n-1$ to compute $h$. 
More generally, this procedure gives $h=n-1 -\sum_{i\ge 3}(i-2)t_{\A,H,i}$ which means $t_{\A,H}=n-1-h$.
\end{proof}

Proposition \ref{FV} gives the following 
criterion to determine the splitting type.

\begin{theorem} \label{main0}
Let $\A$ be a line arrangement in $\P^2$ with $n:=|\A|$. Let $H \in \A$ and 
$L$ be a line not in $\A$. Define 
$h:=|\A^H|$ and $\ell:=|\{L \cap K \mid K \in \A\}|$. Let $d_1\le d_2$ and $e_1\le e_2$ be  integers  such that:
$$T_{\A}|_H\simeq \mathcal{O}_H(-d_1) \oplus \mathcal{O}_H(-d_2) \qquad \mbox{and} \qquad T_{\A}|_L
\simeq \mathcal{O}_L(-e_1) \oplus \mathcal{O}_L(-e_2).$$
Then we have:
\begin{enumerate}[i)]
\item \label{i} if $ n-h \le \lceil \frac{n-1}{2} \rceil$, then 
$(d_1,d_2)=(n-h,h-1)$,
\item \label{ii} if $ n-\ell \ge \lceil  \frac{n}{2} \rceil$, then 
$(e_1,e_2)=(n-\ell,\ell-1)$.
\end{enumerate}
\end{theorem}

\begin{proof}
The assertion \eqref{i} is known 
(see \cite{WY} for example) but not the assertion \eqref{ii}. Since we can give the same kind of proof to these two statements 
by using   Proposition \ref{FV}, we prove both of them.

For \eqref{i}, by Proposition \ref{FV}, there is an 
exact sequence 
$$\begin{CD}
0 @>>> T_\A
@>>> T_{\A \setminus \{H\}}
@>>> \mathcal{O}_H(-t_{\A,H}) @>>> 0,
\end{CD}
$$
where 
$t_{\A,H}=n-1-h$. Take the dual of this sequence and tensor it by $\mathcal{O}_H$ to obtain a surjection
$$\begin{CD}
T_\A^{\vee}|_H @>>> \mathcal{O}_H(t_{\A,H}+1) @>>> 0.
\end{CD}
$$

Since $T_\A^{\vee}|_H 
\simeq \mathcal{O}_H(d_1) \oplus \mathcal{O}_H(d_2)$, the surjection requires 
$t_{\A,H}+1=n-h= d_1$ or $t_{\A,H}+1=n-h\ge  d_2$. 
If $ n-h \le \lceil \frac{n-1}{2} \rceil$, then $n-h\le
d_2$ and therefore  $n-h=d_1$ or $n-h=d_2$. In both cases  we get $(d_1,d_2)=(n-h,h-1)$.
\medskip

Let us now prove \eqref{ii}.
Consider the exact sequence of Proposition \ref{FV}:
$$\begin{CD}
0 @>>> T_\B
@>>> T_\A
@>>> \mathcal{O}_L(-t_{\B,L}) @>>> 0,
\end{CD}
$$
where $\B:=\A \cup \{L\}$ and $t_{\B,L}=n-\ell$. Restricting this onto $L$, we have 
$$\begin{CD}
T_\A|_L \simeq 
\mathcal{O}_L(-e_1)
\oplus 
\mathcal{O}_L(-e_2) 
@>>> \mathcal{O}_L(-n+\ell) @>>> 0.
\end{CD}
$$

The surjection requires 
$n-\ell\le e_1$ or $n-\ell=  e_2$. 
If $ n-\ell \ge \lceil \frac{n}{2} \rceil$, then
$n-\ell\ge e_1$ and therefore  $n-\ell=e_1$ or $n-\ell=e_2$. In both
cases we obtain $(e_1,e_2)=(n-\ell,\ell-1)$.

%

\section{Resolution of the logarithmic bundle}

\label{resolution}

Here we prove Theorem \ref{main1}.
We construct the arrangement $\A$ in three steps, starting from a grid
of horizontal and vertical lines, and
adding two series of diagonal 
lines, in such a way that the resolution of the logarithmic bundle
remains under control.

\subsection{Starting from the grid}

Let us start by defining the lines:\[
X_i = \{x=iz\}, \qquad Y_i = \{y=iz\}, \qquad H_\infty = \{z=0\}.
\]
The \textit{grid} arrangement consists of the line at infinity and $k+2j+1$
``parallel'' lines $X_i$ and $Y_i$:
$$
\A_0=H_\infty \cup \bigcup_{i=-j}^{k+j} (X_i \cup Y_i) =\{z\prod_{i=-j}^{k+j}(x-i z)(y-i z)=0\}.
$$
It is well-known (or we may 
apply \cite[Proposition 3.3]{FV} to show) that $\A_0$ is free and 
we have $T_{\A_0} \simeq \mathcal{O}_{\P^2}(-k-2j-1)^2$.

\begin{figure}[h!]
  \centering
            \begin{tikzpicture}[scale=0.5]
            \draw (1.8,1.8) node {\tiny{$0$}};
            \draw (2.8,1.8) node {\tiny{$1$}};
            \draw (3.8,1.8) node {\tiny{$2$}}; \draw (4.8,1.8) node {\tiny{$3$}};
            \draw (1.8,2.8) node {\tiny{$1$}};
            \draw (1.8,3.8) node {\tiny{$2$}}; \draw (1.8,4.8) node {\tiny{$3$}};
            \draw (1,-0.1) -- (1,5.5);
            \draw[thick,->] (2,-0.3) to (2,6.8);
            \draw (3,-0.1) -- (3,5.5);
            \draw (4,-0.1) -- (4,5.5); \draw (5,-0.1) -- (5,5.5);
            \draw (5,-0.1) -- (5,5.5);
            \draw (0,-0.1) -- (0,5.5);
            \draw (-0.1,1) -- (5.4,1);
            \draw (-0.1,0) -- (5.4,0);
            \draw[thick,->] (-0.2,2) -- (6.2,2);
            \draw (-0.1,3) -- (5.4,3);
            \draw (-0.1,4) -- (5.4,4); \draw (-0.1,5) -- (5.4,5);
            \draw (-0.1,5) -- (5.4,5);
          \end{tikzpicture}
    \caption*{Grid arrangement $\A_0$ for $j=2$, $k=1$.}
\end{figure}
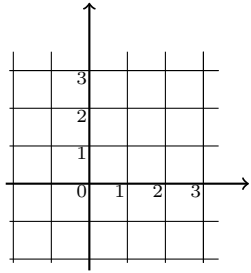

We will now add diagonal lines. 
For $i=-j,\ldots,k+j+1$ we let:
\[
H_i:=\{y+x=(k-i+1)z\}.  
\]

\subsection{Adding diagonal inner lines}

Now we add the diagonal inner lines, namely the lines $H_i$ of $\A$, for
$i=1,\ldots,k+j+1$.
In other words, starting from the grid, we add diagonal lines
 lying $k$ integral steps above the origin and proceed downwards, in such a way that the total number of
triple points along each new line is constant.
Indeed, the decreasing number of affine triple points on $H_i$ is compensated by
the increasing multiplicity at infinity.

Now define,    the nested arrangements
$\A_0 \subset \A_1 \subset \cdots \subset \A_{k+j+1}$ as follows:
$$
\A_i=\A_0 \cup H_1 \cup \cdots \cup H_i.
$$

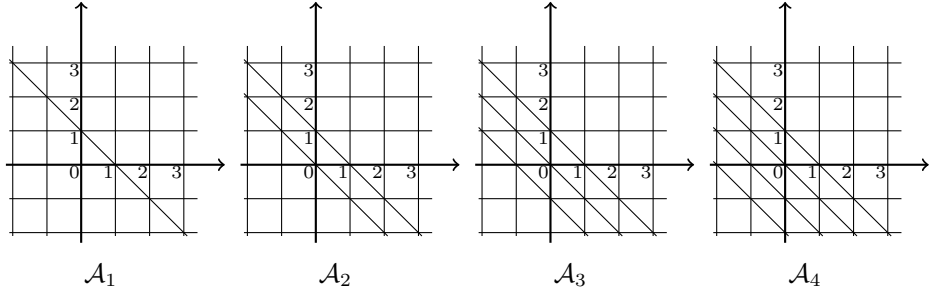
\begin{figure}
\quad
    \begin{subfigure}[b]{0.2\textwidth}
          \begin{tikzpicture}[scale=0.45]
            \draw (1.8,1.8) node {\tiny{$0$}};
            \draw (2.8,1.8) node {\tiny{$1$}};
            \draw (3.8,1.8) node {\tiny{$2$}}; \draw (4.8,1.8) node {\tiny{$3$}};
            \draw (1.8,2.8) node {\tiny{$1$}};
            \draw (1.8,3.8) node {\tiny{$2$}}; \draw (1.8,4.8) node {\tiny{$3$}};
            \draw (1,-0.1) -- (1,5.5);
            \draw[thick,->] (2,-0.3) to (2,6.8);
            \draw (0,-0.1) -- (0,5.5);
            \draw (3,-0.1) -- (3,5.5);
            \draw (4,-0.1) -- (4,5.5); \draw (5,-0.1) -- (5,5.5);
            \draw (-0.1,1) -- (5.4,1);
            \draw (-0.1,0) -- (5.4,0);
            \draw[thick,->] (-0.2,2) -- (6.2,2);
            \draw (-0.1,3) -- (5.4,3);
            \draw (-0.1,4) -- (5.4,4); \draw (-0.1,5) -- (5.4,5);
            \draw (5.1,-0.1) -- (-0.1,5.1);
          \end{tikzpicture}
        \caption*{$\A_1$}   
    \end{subfigure}
\hspace{0.3cm}
    \begin{subfigure}[b]{0.2\textwidth}
        \centering
          \begin{tikzpicture}[scale=0.45]
            \draw (1.8,1.8) node {\tiny{$0$}};
            \draw (2.8,1.8) node {\tiny{$1$}};
            \draw (3.8,1.8) node {\tiny{$2$}}; \draw (4.8,1.8) node {\tiny{$3$}};
            \draw (1.8,2.8) node {\tiny{$1$}};
            \draw (1.8,3.8) node {\tiny{$2$}}; \draw (1.8,4.8) node {\tiny{$3$}};
            \draw (1,-0.1) -- (1,5.5);
            \draw[thick,->] (2,-0.3) to (2,6.8);
            \draw (0,-0.1) -- (0,5.5);
            \draw (3,-0.1) -- (3,5.5);
            \draw (4,-0.1) -- (4,5.5); \draw (5,-0.1) -- (5,5.5);
            \draw (-0.1,1) -- (5.4,1);
            \draw (-0.1,0) -- (5.4,0);
            \draw[thick,->] (-0.2,2) -- (6.2,2);
            \draw (-0.1,3) -- (5.4,3);
            \draw (-0.1,4) -- (5.4,4); \draw (-0.1,5) -- (5.4,5);
            \draw (5.1,-0.1) -- (-0.1,5.1);
            \draw (4.1,-0.1) -- (-0.1,4.1);
          \end{tikzpicture}
        \caption*{$\A_2$}
    \end{subfigure}
\hspace{0.3cm}
    \begin{subfigure}[b]{0.2\textwidth}
        \centering
          \begin{tikzpicture}[scale=0.45]
            \draw (1.8,1.8) node {\tiny{$0$}};
            \draw (2.8,1.8) node {\tiny{$1$}};
            \draw (3.8,1.8) node {\tiny{$2$}}; \draw (4.8,1.8) node {\tiny{$3$}};
            \draw (1.8,2.8) node {\tiny{$1$}};
            \draw (1.8,3.8) node {\tiny{$2$}}; \draw (1.8,4.8) node {\tiny{$3$}};
            \draw (1,-0.1) -- (1,5.5);
            \draw[thick,->] (2,-0.3) to (2,6.8);
            \draw (0,-0.1) -- (0,5.5);
            \draw (3,-0.1) -- (3,5.5);
            \draw (4,-0.1) -- (4,5.5); \draw (5,-0.1) -- (5,5.5);
            \draw (-0.1,1) -- (5.4,1);
            \draw (-0.1,0) -- (5.4,0);
            \draw[thick,->] (-0.2,2) -- (6.2,2);
            \draw (-0.1,3) -- (5.4,3);
            \draw (-0.1,4) -- (5.4,4); \draw (-0.1,5) -- (5.4,5);
            \draw (5.1,-0.1) -- (-0.1,5.1);
            \draw (4.1,-0.1) -- (-0.1,4.1);
            \draw (3.1,-0.1) -- (-0.1,3.1);
          \end{tikzpicture}
        \caption*{$\A_3$}
    \end{subfigure}
    \hspace{0.3cm}
    \begin{subfigure}[b]{0.2\textwidth}
        \centering
          \begin{tikzpicture}[scale=0.45]
            \draw (1.8,1.8) node {\tiny{$0$}};
            \draw (2.8,1.8) node {\tiny{$1$}};
            \draw (3.8,1.8) node {\tiny{$2$}}; \draw (4.8,1.8) node {\tiny{$3$}};
            \draw (1.8,2.8) node {\tiny{$1$}};
            \draw (1.8,3.8) node {\tiny{$2$}}; \draw (1.8,4.8) node {\tiny{$3$}};
            \draw (1,-0.1) -- (1,5.5);
            \draw[thick,->] (2,-0.3) to (2,6.8);
            \draw (0,-0.1) -- (0,5.5);
            \draw (3,-0.1) -- (3,5.5);
            \draw (4,-0.1) -- (4,5.5); \draw (5,-0.1) -- (5,5.5);
            \draw (-0.1,1) -- (5.4,1);
            \draw (-0.1,0) -- (5.4,0);
            \draw[thick,->] (-0.2,2) -- (6.2,2);
            \draw (-0.1,3) -- (5.4,3);
            \draw (-0.1,4) -- (5.4,4); \draw (-0.1,5) -- (5.4,5);
            \draw (5.1,-0.1) -- (-0.1,5.1);
            \draw (4.1,-0.1) -- (-0.1,4.1);
            \draw (3.1,-0.1) -- (-0.1,3.1);
            \draw (2.1,-0.1) -- (-0.1,2.1);
          \end{tikzpicture}
        \caption*{$\A_4$}
    \end{subfigure}
    \caption*{Nested arrangements $\A_i$ for $j=2$, $k=1$.}
\end{figure}

For $i=1,\ldots,k+j+1$, we compute 
\begin{align}
\label{c1} 
\nonumber & |\A_i|-1-|(\A_i \cup \{H_{i+1}\})^{H_{i+1}}|=k+2j+1.
\end{align}
In other words, $t_{\A_i,H_{i+1}}=k+2j+1$ by Lemma \ref{lemma}.
Hence by Proposition \ref{FV}  there is an exact sequence:
\begin{equation}
  \label{Ti-Fi}
0 \longrightarrow T_{\A_{i+1}} \longrightarrow 
T_{\A_i} \longrightarrow \mathcal{O}_{H_{i+1}}(-k-2j-1) \longrightarrow 0.
\end{equation}
 For $i=1,\ldots,k+j+1$, let:
 \[F_i:=T_{\A_i}(k+2j+i).\]
 Then $c_1(F_{i})=i-2$ and it is easy to check that:
$$
F_1  \simeq \mathcal{O}_{\P^2} \oplus \mathcal{O}_{\P^2}(-1), \qquad
F_2  \simeq  \mathcal{O}_{\P^2}^{\oplus 2} \qquad \mbox{and} \qquad
F_3  \simeq  T_{\P^2}(-1).
$$

The sequence of $F_i$ constructed in this way corresponds to the  ``extension step'' explained in \S\ref{steiner-reduction-extension}. Indeed, using 
that $E^\vee \simeq E(-c_1(E))$ for a rank-$2$ bundle $E$, the dual exact sequence of (\ref{Ti-Fi}) gives  
 the following extension for $2\le i\le k+j$: 
\[
0 \longrightarrow F_{i} \longrightarrow F_{i+1} \longrightarrow \mathcal{O}_{H_{i+1}} \longrightarrow 0.  
\]
Then, for $2\le i\le k+j$, by \S \ref{extension} the line $H_{i+1}$ is unstable for $F_{i+1}$, and
$F_{i+1}$ is a Steiner bundle with resolution: 
$$
0 \longrightarrow \mathcal{O}_{\P^2}(-1)^{i-1} \longrightarrow
\mathcal{O}_{\P^2}^{i+1} \longrightarrow 
F_{i+1} \longrightarrow 0.
$$

\subsection{Diagonal outer lines}

Now we add the remaining diagonal lines of $\A$, which we call ``outer''. We start from $H_0$ 
(i.e., the line lying right above $H_1$) and go upwards, i.e. we add
$H_{-1},\ldots,H_{1-j}$.
In other words, for $i = 1,\ldots,j$,  we define the nested arrangements:
$$
\B_i=\A_{k+j+1} \cup H_{0} \cup \cdots \cup H_{1-i}.
$$
Therefore $\A = \B_{j}$. We fix the following notation:
\[
E_0= F_{k+j+1}= T_{\A_{k+j+1}}(2k+3j+1) \qquad \mbox{and} \qquad  E_i=T_{\B_i}(2k+3j+1).
\]
\begin{figure}[h!]
    \begin{subfigure}[b]{0.2\textwidth}
        \centering
          \begin{tikzpicture}[scale=0.5]
            \draw (1.8,1.8) node {\tiny{$0$}};
            \draw (2.8,1.8) node {\tiny{$1$}};
            \draw (3.8,1.8) node {\tiny{$2$}}; \draw (4.8,1.8) node {\tiny{$3$}};
            \draw (1.8,2.8) node {\tiny{$1$}};
            \draw (1.8,3.8) node {\tiny{$2$}}; \draw (1.8,4.8) node {\tiny{$3$}};
            \draw (1,-0.1) -- (1,5.5);
            \draw[thick,->] (2,-0.3) to (2,6.8);
            \draw (0,-0.1) -- (0,5.5);
            \draw (3,-0.1) -- (3,5.5);
            \draw (4,-0.1) -- (4,5.5); \draw (5,-0.1) -- (5,5.5);
            \draw (-0.1,1) -- (5.4,1);
            \draw (-0.1,0) -- (5.4,0);
            \draw[thick,->] (-0.2,2) -- (6.2,2);
            \draw (0.9,5.1) -- (5.1,0.9);
            \draw (-0.1,3) -- (5.4,3);
            \draw (-0.1,4) -- (5.4,4); \draw (-0.1,5) -- (5.4,5);
            \draw (5.1,-0.1) -- (-0.1,5.1);
            \draw (4.1,-0.1) -- (-0.1,4.1);
            \draw (3.1,-0.1) -- (-0.1,3.1);
            \draw (2.1,-0.1) -- (-0.1,2.1);
          \end{tikzpicture}
        \caption*{$\B_1$}
    \end{subfigure}
\hspace{2cm}
    \begin{subfigure}[b]{0.4\textwidth}
        \centering
          \begin{tikzpicture}[scale=0.5]
            \draw (1.8,1.8) node {\tiny{$0$}};
            \draw (2.8,1.8) node {\tiny{$1$}};
            \draw (3.8,1.8) node {\tiny{$2$}}; \draw (4.8,1.8) node {\tiny{$3$}};
            \draw (1.8,2.8) node {\tiny{$1$}};
            \draw (1.8,3.8) node {\tiny{$2$}}; \draw (1.8,4.8) node {\tiny{$3$}};
            \draw (1,-0.1) -- (1,5.5);
            \draw[thick,->] (2,-0.3) to (2,6.8);
            \draw (0,-0.1) -- (0,5.5);
            \draw (3,-0.1) -- (3,5.5);
            \draw (4,-0.1) -- (4,5.5); \draw (5,-0.1) -- (5,5.5);
            \draw (-0.1,1) -- (5.4,1);
            \draw (-0.1,0) -- (5.4,0);
            \draw[thick,->] (-0.2,2) -- (6.2,2);
            \draw (0.9,5.1) -- (5.1,0.9);
            \draw (1.9,5.1) -- (5.1,1.9);
            \draw (-0.1,3) -- (5.4,3);
            \draw (-0.1,4) -- (5.4,4); \draw (-0.1,5) -- (5.4,5);
            \draw (5.1,-0.1) -- (-0.1,5.1);
            \draw (4.1,-0.1) -- (-0.1,4.1);
            \draw (3.1,-0.1) -- (-0.1,3.1);
            \draw (2.1,-0.1) -- (-0.1,2.1);
          \end{tikzpicture}
        \caption*{$\B_2 = \A$}
    \end{subfigure}
\caption*{Arrangements $\B_i$ for $j=2$, $k=1$.}
\end{figure}
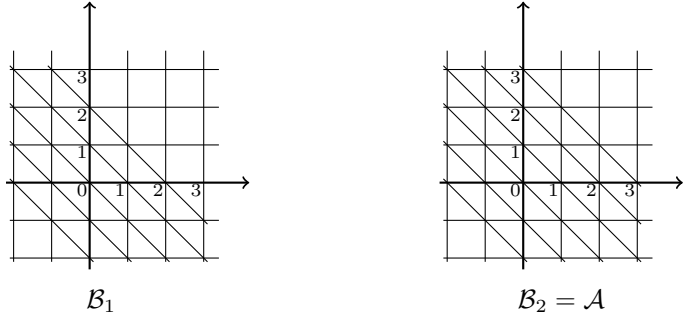

For $i=1,\ldots,j$, it is easy to compute
$
t_{\B_i,H_{1-i}}=2k+3j+1$. By Theorem \ref{main0}, we have thus the exact sequences:
$$
0 \longrightarrow 
E_i \longrightarrow E_{i-1} \longrightarrow \mathcal{O}_{H_{1-i}} \longrightarrow 
0.
$$
The line $H_{1-i}$ is then unstable for $E_{i-1}$. So  the ``reduction step'' recalled in \S\ref{reduction},
implies that 
$E_i $ is a Steiner bundle, for $i=1,\ldots,j$,  with resolution
of the form:
$$
0 \longrightarrow 
\mathcal{O}_{\P^2}^{k+j-1-i}(-1) \longrightarrow
\mathcal{O}_{\P^2}^{k+j+1-i} \longrightarrow 
E_i \longrightarrow
0.
$$
The case $i=j$ completes the proof of Theorem \ref{main1}. 
\end{proof}

\begin{rem}
\label{def-extended}
If we extend the definition of $H_{1-i}$ and $\B_i$ to the range
$i=j+1,\ldots,k+j-2$, the equality $t_{\B_i,H_{1-i}}=2k+3j+1$ remains valid.
Applying repeatedly the reduction step \S\ref{reduction}, 
we get $E_{k+j-2} \simeq T_{\P^2}(-1)$.
Also, $W(E_{k+j-3})$ is a smooth
conic $C \subset \PD^2$.
Also, we will see that 
$E_{k+j-4}$ has exactly six jumping lines, and that only five of them
lie on $C$. 
We will use these supplementary bundles in the next section.
\end{rem}

\section{Unstable lines of the logarithmic bundle}

\label{unstable}

Let us fix again the notation as in the introduction:
$$
\A=\A_{A_2}^{[-j,k+j]}.
$$
We have just verified that 
$T_{\A}$ is a Steiner bundle. We can 
now investigate its splitting type in more detail. The following is known, see \cite{W}  for 
example. 

\begin{prop} 
The splitting type of $\A$ at infinity is:
\[
T_{\A}(2k+3j+1)|_{H_\infty} \simeq 
\mbox{$\mathcal{O}_{H_\infty}(\lceil  \frac{k-1}{2} \rceil) \oplus \mathcal{O}_{H_\infty}(\lfloor  \frac{k+1}{2} \rfloor)$.}
\]
\end{prop}
In other words, the splitting type of $T_{\A}$ onto $H_\infty$ is 
generic. Now let us look for jumping lines of
$T_{\A}$. Recall that $c_1(T_{\A}(2k+1+3j))=k-1$. Then 
a line $L \subset \P^2$ is a jumping line of $T_\A$
if and only if:
\[
H^0(\P^2,T_\A(2k+3j+1-s)|_L) \neq 0, \qquad \mbox{for $2s\ge k+1$},
\]
or equivalently, by Serre duality, if and only if:
$$
H^1(\P^2,T_\A(k+3j+s)|_L) \ne 0 ,\qquad \mbox{for $2s\ge k+1$}.
$$
The jumping lines are unstable lines when $s=k-1$.

\begin{prop} 
\label{jump}
Let $0\le 2s \le k-1$ and let $H$ be one of the following lines:
\begin{align*}
 &\mbox{lines of $\A$:} && X_{k+j-s}, && Y_{k+j-s}, &&H_{k+j+1-s},\\
 &\mbox{lines outside $\A$:} &&X_{-(j+s+1)}, && Y_{-(j+s+1)}, &&H_{-(j+s)}.
\end{align*}
Then $T_\A(2k+3j+1)|_H \simeq \mathcal{O}_H(k-1-s) \oplus 
\mathcal{O}_H(s)$.
\end{prop}

\begin{proof}
Let us prove the statement for the lines of $\A$.
It is easy to check that $t_{\A,H}=k+3j+1+s$. 
Hence applying Theorem \ref{main0}, part \eqref{i}, we know that:
\begin{equation}
  \label{splittingtype}
T_\A|_H \simeq 
\mathcal{O}_H(-2k-3j-1+s) \oplus \mathcal{O}_H(-k-3j-2-s).  
\end{equation}

Let us check the lines outside $\A$.
It is easy to see that $t_{\A \cup \{H\},H}=2k+3j+1-s$. 
Hence applying Theorem \ref{main0}, part \eqref{ii}, we get again \eqref{splittingtype}.
\end{proof}

\begin{cor} \label{unstable1}
Let $\kappa \ge 0$ and $s \ge 0$ be integers.
Then the set of jumping lines of order $\kappa+s-j$ of $T_\A$ contains:
\begin{align*}
 &X_{2\kappa +s}, && Y_{2\kappa +s}, && H_{k+s+1},\\  
 &X_{(-2j+s+1)}, && Y_{(-2j+s+1)}, && H_{k+s-2\kappa-2j}.
\end{align*}
In particular, the following lines are unstable 
 for $T_{\A}$:
\begin{align*}
  &X_{k+j}, && Y_{k+j}, && H_{k+j+1},\\ 
  &X_{-(j+1)},&& Y_{-(j+1)},&& H_{-j}.
\end{align*}
\end{cor}

\begin{rem}
 These lines do not cover 
the whole set of jumping lines of any given  order, even when the
first Chern class is odd (in which case some of these sets are
presumably finite). 
These lines are determined by the 
arrangement, but it seems very difficult to determine all of them. 
We comment on this shortly in one example, see \S\ref{experiments}.
\end{rem}
\begin{proof}[Proof of Theorem \ref{main3}]
In order to prove Theorem \ref{main3}, we have to show that 
the six lines appearing in Corollary \ref{unstable1} are tangent to
$C_j$ for $k=3$ and that these six lines form 
the set of all unstable lines of $T_{\A}$ for $k\ge 4$. The first fact is clear
from the picture. To prove the second one, we will use the step-by-step construction like in the proof of 
Theorem \ref{main2} whereby keeping the same notation.

Recall that
$H_{i+1}$ is an unstable line for the Steiner bundle $F_{i+1}$.
Now computing the number of triple points  and using Theorem
\ref{main0} we can see that the following $5$ lines are also unstable
for $F_{i+1}$:
\begin{equation} \label{5lines}
  \begin{array}[h]{lll}
  X_{k+j}, & Y_{k+j}, & H_0,\\ 
  X_{-(j+1)},& Y_{-(j+1)}.
  \end{array}
\end{equation}
Let us call $C_{k,j}$ the conic tangent to these last $5$ lines.
The line $H_4$ is tangent to $C_{k,j}$, indeed $W(F_4)$ consist of all lines tangent to
$C_{k,j}$.
Since $H_{5}$ is parallel both to $H_4$ and $H_0$, it is not tangent to $C_{k,j}$.
Then, $W(F_5)$ consists of the $5$ lines of \eqref{5lines}, plus $H_5$. In other words: 
\[
W(F_{5}) \setminus \{H_5\} \subset W(F_{i+1}).
\]

Applying the procedure of \S\ref{extension} repeatedly, we see that:
\[
W(F_{i+1}) \setminus \{H_{i+1}\} \subset W(F_i),
\]
for $i=5, \ldots, k+j$.
Using Proposition \ref{jump}, we see that
the line $H_i$ is not  unstable for $F_{i+1}$.  Then we have more precisely, for $i=5, \ldots, k+j$, 
\[
W(F_{i+1}) \setminus \{H_{i+1}\} \subset W(F_i) \setminus \{H_{i}\}.
\]
The line $H_{k+j+1}$ is of course not unstable for $F_5$. Then putting all these inclusions together, we find:
\[
W(F_{5}) \setminus \{H_5\} \subset W(F_{k+j+1}) \setminus \{H_{k+j+1}\} \subset W(F_5) \setminus \{H_{5}\}.
\]

We have thus proved:
\[
W(F_{k+j+1}) = \{H_{k+j+1}\} \cup W(F_5) \setminus \{H_5\}.
\]

Now we consider the ``outer'' construction. Computing the number of triple points  and using Theorem
\ref{main0} we can see that $H_{-i}$ and the following $5$ lines are  unstable
for $E_{i}=T_{\B_{i}}(2k+3j+1)$:
\begin{equation} \label{5lines2}
  \begin{array}[h]{lll}
  X_{k+j}, & Y_{k+j}, & H_{k+j+1},\\ 
  X_{-(j+1)},& Y_{-(j+1)}.
  \end{array}
\end{equation}
This happens for $i=1, \ldots, j$, but also for $i=j+1, \ldots, k+j-2$, as it was observed in Remark \ref{def-extended}.
Let us call $\Gamma_{k,j}$ the conic tangent to these last $5$ lines.
The line $H_{-k-j+3}$ is tangent to $\Gamma_{k,j}$, as $W(E_{k+j-3})$ consists of all lines tangent to
$\Gamma_{k,j}$.
Since $H_{-k-j+4}$ is parallel both to $H_{-k-j+3}$ and $H_0$, it is not tangent to $\Gamma_{k,j}$.
Therefore, $W(E_{k+j-4})$ consists of the $5$ lines of
\eqref{5lines2}, and of the extra line $H_{-k-j+4}$.

In other words: 
\[
W(E_{k+j-4}) \setminus \{H_{-k-j+4}\} \subset W(E_{j}).
\]

Applying the procedure of \S\ref{extension} repeatedly, we see that:
\[
W(E_{i}) \setminus \{H_{-i}\} \subset W(E_{i+1}),
\]
for $i=1, \ldots, k+j-4$.
Using Proposition \ref{jump}, we see that
the line $H_{-i-1}$ is not  unstable for $E_{i}$.  Then we have more precisely, for $i=1, \ldots, k+j-4$, 
\[
W(E_{i}) \setminus \{H_{-i}\} \subset W(E_{i+1})\setminus \{H_{-i-1}\}
\]
The line $H_{-j}$ is of course not unstable for $E_{k+j-4}$. Then putting all these inclusions together, we find:
\[
W(E_{k+j-4}) \setminus \{H_{-k-j+4}\} \subset W(E_{j})\setminus \{H_{-j}\} \subset W(E_{k+j-4}) \setminus \{H_{-k-j+4}\}.
\]

This proves:
\[
W(E_{j}) = \{H_{-j}\} \cup W(E_{k+j-4}) \setminus \{H_{-k-j+4}\}.
\]
The proof of the theorem is now complete.\end{proof}

\begin{proof}[Proof of Corollary \ref{non}]
When $k=0,1,2$, the isomorphism (\ref{eq2}) is shown in \cite{Y0} and \cite{A}. 
Assume that $k \ge 3$. Then Theorem \ref{main3} shows that the unstable lines 
vary when $j$ varies, which completes the proof.
\end{proof}

\section{Experimental results}
\label{experiments}

We outline here some experimental results concerning logarithmic
bundles associated with root systems of type $A_m$.
First, we give a first estimate on the behaviour of the set of jumping
lines for $m=2$. Then, we provide a counter-example to the dual shift
isomorphism in case $m=3$.

\subsection{Jumping lines}

An interesting question is to determine the set of all jumping lines
of $T_\A$ when $\A=\A_{A_2}^{[-j,k+j]}$ and not only the unstable ones.
The first open case is $j=0$ and $k=6$, so that the arrangement $\A$ consists of
$22$ lines, and $T_\A$ has the resolution:
\[
0 \to \mathcal{O}_{\P^2}(-1)^5 \to  \mathcal{O}^7_{\P^2} \to T_\A(13) \to 0.
\]
The set $S(T_\A)$ of jumping lines is the support in $\PD^2$ of the cokernel of a map:
\[
(\mathrm{Sym}^2(\Omega_{\PD^2}(1)))^5 \to \Omega_{\PD^2}(1)^7.
\]
Combining with the Euler sequence and taking syzygies, $S(T_\A)$ is
seen to be the locus cut by the $8 \times 8$ minors of a matrix of the
form:
\[
\mathcal{O}_{\PD^2}^9(-1) \to \mathcal{O}^8_{\PD^2}.
\]
We may determine this matrix (defined over $\mathbb{Q}$) explicitly with Macaulay2, and find that
that $S(T_\A)$ has length $36$. It contains $6$ triple points, one
for each unstable line, and $6$ smooth points corresponding to the $6$
jumping lines of Corollary \ref{unstable1}.
Moreover it contains $12$ (smooth) points with irrational coordinates
which are a complete intersection of a quartic and a cubic.
These curves both have rational (or integral) coefficients,  and can also
be determined explicitly. For instance the cubic is:
\[
\small{\mbox{$62x^3-90x^2y-90x y^2+62y^3+x^2z-109x y z+y^2z-18x z^2-18y z^2-3z^3$.}}
\]
\subsection{Failure of dual shift isomorphism}

In $\P^m$ with $m \ge 3$ the notion of unstable line for a Steiner
sheaf $E$ lying in:
\[
0 \to \mathcal{O}_{\P^m}(-1)^\ell \to \mathcal{O}_{\P^m}^{\ell+m}\to E \to 0 
\]
is replaced by the idea  of \textit{unstable hyperplane}, namely $H$ is so if:
\[
H^{m-1}(H,E|_H(-m)) \ne 0.
\]

Let $\A=\A_{A_3}^{[0,2]}$. With Macaulay2 we get
the resolution:
\[
0 \to \mathcal{O}_{\P^3}(-8)^{3} \to \mathcal{O}_{\P^3}(-7)^{6}\to T_\A \to 0,
\]
and check that $T'=T_\A^\vee(-12)$ has the same Betti numbers of $T_\A$, i.e. $6$
generators of degree $7$ and $3$ linear syzygies. This confirms
Conjecture \ref{all}.

The sets of unstable planes of $E=T_\A(7)$ and
$E'=T'(7)$ are computed as loci cut by maximal minors, respectively, of
matrices $N$ and $N'$:
\[
N, N' : \mathcal{O}_{\PD^3}(-1)^{18} \to \mathcal{O}_{\PD^3}^{6}.
\]
Macaulay2 here says that, although both $N$ and $N'$ have a cokernel of length
$7$ consisting of $7$ smooth points (so that $E$ and $E'$ both have
$7$ unstable distinct planes), the underlying sets of planes are different, actually only $4$
planes are common to these sets. So $E$ and $E'$ cannot be isomorphic.
This shows that the dual shift isomorphism \eqref{dualiso} does not take place in general.

Note that, for instance for $\Phi=A_4$, one can choose $k$ and $j$ in
such a way that $T_\A$ and $T_\A^\vee(-\eta(k+2j+1))$ do not even have
the same number of unstable hyperplanes.


\begin{thebibliography}{Ath0}

\bibitem[A]{A}
T. Abe, 
The stability of the family of $A_2$-type arrangements. 
\textit{J. Math. Kyoto Univ}. Vol. 46 (2006), no. 3, 617--636.

\bibitem[Ath0]{Ath0}
C. Athanasiadis, 
Deformations of Coxeter hyperplane arrangements 
 and their characteristic polynomials. In \textit{Arrangements - Tokyo 1998}. 1--26. 
 Advanced Studies in Pure Mathematics \textbf{27}, Kinokuniya, Tokyo, 2000.

 \bibitem[Ath1]{Ath1}
 C. Athanasiadis, 
 Extended Linial hyperplane arrangements for root systems and a conjecture of Postnikov and Stanley.
 {\it J. Algebraic Combin.} {\bf 10} (1999), no. 3, 207–225. 


\bibitem[DK]{DK} I. Dolgachev and  M. Kapranov, Arrangements of hyperplanes and vector bundles on
${\bf P}_{n}$, \textit{Duke Math. J.} {\bf 71}, (1993), 633-664.

\bibitem[ER]{ER}
P. Edelman and V. Reiner, 
Free arrangements and rhombic tilings. 
\textit{Discrete and Comp.
Geom}., {\bf 15} (1996), 307--340.


\bibitem[FMV]{FMV}
D. Faenzi, D. Matei and J. Vall\`es,
Hyperplane arrangements of Torelli type.
\textit{Compositio
Math.}, \textbf{149} (2013), no. 2, 309--332. 

\bibitem[FV]{FV}
D. Faenzi and J. Vall\`es,
Logarithmic bundles and Line arrangements, an approach via the standard construction.
arXiv:1209.4934. 


\bibitem[M2]{M2}
D. R. Grayson  and M. E. Stillman,
Macaulay2, a software system for research in algebraic geometry.
Available at \url{http://www.math.uiuc.edu/Macaulay2/}




\bibitem[OSS]{OSS} C. Okonek, M. Schneider, H. Spindler, \textit{Vector Bundles on
Complex Projective Spaces}. \textbf{3} (1980), Birkh\"auser.



\bibitem[PS]{PS}
 A. Postnikov and R. P. Stanley, Deformations of Coxeter hyperplane arrangements. 
 \textit{J. Combin. Theory Ser. A} \textbf{91} (2000), no. 1-2, 544--597.


\bibitem[V]{V}
 J. Vall\`es,
Nombre maximal d'hyperplans instables pour un fibr\'e de Steiner. 
\textit{Math. Z.} \textbf{233} (2000), no. 3, 507--514. 
arXiv:1209.4934. 


\bibitem[W]{W} A. Wakamiko, On the Exponents of $2$-Multiarrangements. 
\textit{Tokyo J. Math.} \textbf{30} (2007), no. 1, 99--116.


\bibitem[WY]{WY}
M. Wakefield and S. Yuzvinsky, Derivations of an effective divisor 
on the complex projective line. \textit{Trans. Amer. Math. Soc.} \textbf{359} (2007), 
no. 9, 4389--4403. 


\bibitem[Y0]{Y0}
M. Yoshinaga,
Characterization of a free arrangement and
conjecture of
Edelman and Reiner. \textit{Invent. Math.} \textbf{157} (2004), no. 2,
449--454.



\bibitem[Y1]{Y}
M. Yoshinaga, 
Freeness of hyperplane arrangements and related topics.
\textit{Ann. Fac. Sci. Toulouse Math.} (6),  \textbf{23} (2014), no. 2, p. 483--512.




\end{thebibliography}
\end{document}